\documentclass[12pt,a4paper]{amsart}
\usepackage{amssymb}
\usepackage[all]{xy}

\usepackage{amsmath,amsfonts,amssymb,amsthm}
\usepackage[all]{xy}

\theoremstyle{definition}
\newtheorem{teo}{Theorem}[]
\theoremstyle{definition}

\theoremstyle{definition}
\newtheorem{lem}[teo]{Lemma}
\theoremstyle{definition}

\theoremstyle{definition}

\theoremstyle{remark}
\newtheorem{rem}{Remark}

\def\R{\mathbb{R}}
\def\C{\mathbb{C}}

\def\K{\mathbb{K}}

\def\P{\mathbb{P}}
\def\Q{\mathbb{Q}}

\def\L{\mathbb{L}}

\begin{document}

\title{Birationally trivial real smooth cubic surfaces}

\author{Jon Gonz\'alez-S\'anchez} 
\address{Departmento de
  Matem\'aticas, Estad\'{\i}stica y Computaci\'on, Facultad de
  Ciencias, Universidad de Cantabria, Avda. de los Castros, E-39071
  Santander, Spain }

\email{jon.gonzalez@unican.es}

\author{Irene Polo Blanco} 
\address{Departmento de
  Matem\'aticas, Estad\'{\i}stica y Computaci\'on, Facultad de
  Ciencias, Universidad de Cantabria, Avda. de los Castros, E-39071
  Santander, Spain }

\email{irene.polo@unican.es}

  \thanks{The first author acknowledges support by the 
Spanish Ministerio de Ciencia e Innovaci\'on, grant MTM2008-06680-C02-01.}

\begin{abstract}
Smooth real cubic surfaces are birationally trivial (over $\R$) if and only if their real locus is connected or, equivalently, if and only if
they have two skew real lines or two skew complex conjugate lines. In such a case a parametrization over the reals 
can be given by cubic polynomials. In this short note we provide a simple geometric method to obtain such parametrization 
based in an algorithm by I. Polo-Blanco and J. Top \cite{Po-Top}.
\end{abstract}

\keywords{Cubic surfaces,  parametrization}
\subjclass{14J26, 14Q10, 65D17}

\maketitle

\section*{Introduction}
A cubic surface $S$ is the vanishing set of a homogeneous polynomial $f$ of degree $3$   
in $\P^3$, i.e.,
\begin{equation*}
S=\{ (x:y:z:t)\in \P^3\mid f(x:y:z:w)=0\}. 
\end{equation*}
By a classical result of Clebsch \cite{Cl} we know that a smooth cubic surface 
over the complex numbers 
admits a parametrization by cubic polynomials defined  over $\C$ (see Theorem \ref{clebsch} below). 
If $S$ is a cubic surface defined over the real numbers, then $S$ admits a parametrization 
defined over the real numbers if and only if the real locus of $S$ is connected 
or, equivalently, if and only if $S$ has two disjoint real lines or two disjoint complex conjugate lines 
(see Theorem \ref{real} below).

Algorithms for parametrizing smooth cubic surfaces have shown to be of great interest in 
Computer Aided Geometric Design  since these surfaces are much  more flexible 
than quadrics and the resolution of many problems in science and technology depends on 
using the so called $A$--cubic patches \cite{Baj0} (bounded and nonsingular cubic surfaces in $\R^3$) to 
approximate certain objects (such as molecules, organs of the human body, etc.). Some applications of 
these particular cubic surfaces ($A$--cubic patches) to molecular 
modeling can be found, for example, in \cite{Baj2,Baj3}. 
Some of the known algorithms  for parametrizing cubic surfaces can be found in the works of 
T. W. Sederberg and J. P.  Snively \cite{Sed}, C. Bajaj, R. Holt and A. Netravali \cite{Baj}, 
T. G. Berry and  R. R. Patterson \cite{Ber} or I. Polo-Blanco and J. Top \cite{Po-Top}.

In this note we parametrize smooth cubic surfaces by giving a purely geometric and 
very simple construction based on the algorithm 
presented in \cite{Po-Top}. The main advantage of this new approach is that  it is easy 
to understand and easy to program, as we will show at the end of this 
note. This a is a small, but we believe important contribution to the theory of real cubic surfaces, that 
might  interest the mathematical community in general and in particular the CAGD researchers.
The construction presented here is, from the computational point of view, optimal. 
It requires the computation of the lines on a smooth cubic surface, a problem recently 
solved by R. Pannenkoek in his master thesis (compare \cite[\S 3.1]{Po-Top} and \cite{Pa}).
However it does not require a different construction for the cases where the cubic surface has two disjoint real lines or 
two disjoint complex conjugate lines, as the previous algorithms did (see \cite{Baj}, \cite{Ber}, \cite{Po-Top}).
We have implemented the algorithm in the mathematical software MAPLE and applied it to the Fermat cubic surface 
and to a generic smooth cubic surface (see Section \ref{last}).
\section{Cubic surfaces}



In 1949 A. Cayley and G. Salmon discovered that smooth cubic surfaces over $\C$ 
have precisely $27$ lines over $\C$. The $27$ lines in a smooth cubic surface have a very special symmetry and 
configuration  and whole books have been written to describe them (e.g. \cite{He}, \cite{Se}). 
L. Schlafli introduced the concept of a \textit{double six}  in order to describe the 
intersection behaviour 
of the $27$ lines in a smooth cubic surface. A double six is a set of 12 of
the 27 lines on a cubic surface $S$, represented with Schl\"afli's
notation as:
$$\left(
\begin{array}{cccccc}
a_1 & a_2 & a_3& a_4& a_5& a_6\\
b_1 & b_2 & b_3& b_4& b_5& b_6
\end{array}\right),$$
with the following intersection behaviour: $a_i$ (resp. $b_i$) does not intersect 
$a_j$ (resp. $b_j$) for $i \neq j$ and $a_i$ intersects
$b_j$ if and only if $i \neq j$. The plane spanned by two intersecting lines $a_i$ and 
$b_j$ for $i \neq j$ intersects the cubic surface on a third line $c_{ij}$. 
Repeating this process the remaining $15$ lines on $S$ are obtained together with their  intersecting behaviour. 

The following property regarding the lines on $S$ will be needed in what follows, and can be easily verified 
by using the configuration of the $27$ lines on $S$.

\begin{lem}\label{combinatoria}
Let $S$ be a nonsingular cubic surface and $\ell_1$ and $\ell_2$ are two skew lines on $S$. 
Then there are precisely $5$ lines that intersect both $\ell_1$ and $\ell_2$. 
\end{lem}

\begin{proof}
See \cite[Lemma 1.2]{Po-Top}
\end{proof}

Another important result concerning smooth cubic surfaces was obtained by 
A. Clebsch in \cite{Cl}.

\begin{teo}[Clebsch]
\label{clebsch}
Let $S\subseteq \P^3(\C )$ be a nonsingular cubic surface. Then $S$ is obtained by mapping 
$\P^2$ to $\P^3$ by the space of cubic forms passing through $6$ points in general position 
(i.e. no three on a line and no six on a conic).
\end{teo}

\begin{proof}
See \cite[Proposition IV.12]{Be}.
\end{proof}



A base change is just a projective transformation of the cubic surface.

\section{Real cubic surfaces}

If $S$ is a real cubic surface then it is not true in general that the morphism 
in Theorem \ref{clebsch} can be defined over the reals. 
For example, the cubic surface defined by the following equation is not 
parametrizable since its real locus is not connected (see \cite{Po-Top1}),
\begin{equation}
X^3 + Y^3 + Z^3 + X^2Y + X^2Z + XY^2 + Y^2Z - 6Y^2W + 11ZW^2 - 6W^3 = 0.
\end{equation}
One has the following characterization of birationally trivial real cubic surfaces.

\begin{teo} \label{real} 
Let $S$ be a smooth cubic surface defined over $\R$. Then the following conditions are equivalent
\begin{enumerate}
\item The real locus $S (\R )$ is connected.
\item $S$ has a set of two skew lines defined over $\R$ (i.e. either both real or complex conjugate).
\item $S$ is birationally trivial over $\R$ (i.e. $S$ is parametrizable over $\R$).
\item $S$ is the blow up of $\P^2$ at six points defined over $\R$ (i.e. $S$ admits a parametrization by cubic polynomials).
\end{enumerate}
\end{teo}
 
\begin{proof}
Compare \cite{Si}, \cite{Ko}, \cite{KM} and \cite{Po-Top1}.
\end{proof}

\section{The parametrization}
\label{s:construction}

The geometric construction we propose works as follows. Let $S$ be a smooth cubic surface defined over $\C$ 
and $\ell_1$ and $\ell_2$ two skew lines on $S$ (i.e. $\ell_1\cap \ell_2=\emptyset$). Consider $m$ a line on $S$ that intersects both 
$\ell_1$ and $\ell_2$ and $H$ a plane containing $m$ that does not contain $\ell_1$ or $\ell_2$. Then for any point $x\in H$ 
there exists a unique line $\ell_x$ passing through $x$ and intersecting both $\ell_1$ and $\ell_2$. The line $\ell_x$ intersects 
the smooth cubic surface at three points:  namely $\ell_x\cap \ell_1$, $\ell_x\cap \ell_2$, and a third point that we will denote by  $q_x$. Now, the parametrization is given by the map
\begin{eqnarray*}
\Phi :H\cong \P^2\to S\subseteq \P^3
\end{eqnarray*}
defined by $\Phi (x)=q_x$.

A real smooth cubic surface $S$ admits a parametrization defined over $\R$ if and only if there exists a 
pair of skew lines on $S$ (call them $\ell_1$ and $\ell_2$) both real or complex conjugated (see Theorem \ref{real}). 
Let us show that, for such $S$, the parametrization $\Phi$ defined above can be constructed over the real numbers. 
Indeed, there are precisely five lines on $S$ intersecting both $\ell_1$ and $\ell_2$ (see Proposition \ref{combinatoria}) 
where at least one of them, call it $m$, is real. This implies that the plane $H$ can be chosen to be real, and therefore, the map $\Phi$ 
will be defined over $\R$.

We continue with the proof that the construction proposed above gives a parametrization of $S$ by cubic polynomials. 
The proof is purely algebraic and it is not needed for the application of the geometric construction to concrete examples. 

\begin{teo}
\label{construction}
Following the above notation, the map $\Phi :\P^2\to S\subseteq \P^3$ is a birational morphism given by cubic polynomials.
\end{teo}

\begin{proof}
The morphism $\Phi$ factor as follows:
$$\xymatrix{H\ar@{-->}[r]^{\Phi_1\ \ \ }&\ell_1\times \ell_2\ar@{-->}^{\ \ \ \Phi_2}[r]&S,}$$
where $\Phi_1(x)=(\ell_x\cap \ell_1,\ell_x\cap \ell_2)\in \ell_1\times \ell_2$ and
$\Phi_2(\ell_x\cap \ell_1,\ell_x\cap \ell_2)=q_x$. The morphism $\Phi_1$ blows up the points 
$\ell_1\cap H$ and $\ell_2\cap H$ and blows down the line $m$. The morphism
$\Phi_2^{-1}$ blows down the five lines that intersect both $\ell_1$ and $\ell_2$, 
in particular the line $m$. Therefore $\Phi=\Phi_2\circ \Phi_1$ blows up $H$ at six points. 
Hence $\Phi$ is given by cubic polynomials.
\end{proof}

\begin{rem}
In the last section of this text we give a computational proof of this fact using the 
program MAPLE. 
\end{rem}

\begin{rem}
\begin{itemize}
\item[(a)] If both lines $\ell_1$ and $\ell_2$ are  defined over $\R$ and one considers 
the map $\Phi_2 :\ell_1\times \ell_2\cong \P^1\times \P^1\longrightarrow S$ in the proof of the previous theorem,
then $\Phi_2$ is a parametrization of $S$ by biquadratic polynomials. In fact $\Phi_2$ is defined by the space of bihomogeneus 
forms of bidegree $(2,2)$ passing through $5$ points $q_1, \ldots ,q_5$ in general possition in $\P^1\times \P^1$. 
This is the algorithm proposed in 
\cite{Sed} and \cite{Baj} with a slight modification for the case when $S$ contains two complex conjugate disjoint lines.
The five points $q_1, \ldots ,q_5$ are the base points of the parametrization.
\item[(b)] If the plane $H$ in the construction above does not contain a line in $S$, then the map $\Phi$ is given by 
polynomials of degree $5$. In fact, in such a  case $\Phi$ is defined by the space of forms of degree $5$ vanishing with multipicity two 
at five points $q_1, \ldots ,q_5$ and with multiplicity one at two points $p_1$ and $p_2$ such that $q_1,\ldots ,q_5$, $p_1$ and $p_2$ 
are in general possition. The seven points $ q_1, \ldots ,q_5, p_1,p_2$ are the base points of the parametrization. 
\end{itemize}
\end{rem}

\section{Examples}
\label{last}

We will now apply the construction given in the previous section to 
some concrete examples. We start with the Fermat cubic.

\subsection{Fermat cubic}

The Fermat cubic surface $S$ in 
$\R^3$ is defined by $1+x_1^3+x_2^3+x_3^3=0$. Following the notation of the previous section 
we take $\ell_1=(-\omega^2,-\omega t,t)$ and 
$\ell_2=(-\omega , -\omega^2 t,t)$ where  $\omega$ is a primitive third root of unity 
\cite[\S 5.2]{Po-Top}. The line $m$ is $(-t,-1,t)$ and we choose the plane $H$ to be  $x_2=-1$.
Now the algorithm that parametrizes  $S$ can be easily implemented in a computer algebra program 
(e.g. in MAPLE).
\begin{quotation}
\sffamily
\noindent
\# Select the Fermat cubic: \\
\noindent
f:=(x1,x2,x3)$\to$ 1+x1\^{}3+x2\^{}3+x3\^{}3: \\
\noindent
w:=RootOf(x\^{}2+x+1,x): \\
\noindent
\# Choose the lines $\ell_1$ and $\ell_2$ on $S$: \\
\noindent
line1 :=r $\to$ [-w\^{}2, -w*r, r]: \\
\noindent
line2 := s $\to$ [-w, -w\^{}2*s, s]:  \\
\noindent
\# The line $\ell_x$: \\
\noindent
linex := t $\to$ [y1,-1,y2]+[t, t*b,t*c]:  \\
\noindent
Sol := solve($\{$line1(r)[1] = linex(t1)[1], line1(r)[2] = linex(t1)[2], line1(r)[3] = linex(t1)[3], line2(s)[1] = linex(t2)[1], line2(s)[2] = linex(t2)[2], line2(s)[3] = linex(t2)[3], f(linex(t3)[1], linex(t3)[2], linex(t3)[3]) = 0, t1 $<>$ t3, t2 $<>$ t3$\}$, [r, s, b, c, t1, t2, t3]): \\
\noindent 
\# The pametrization of $S$ is \\
\noindent
parametrization := factor(subs(Sol[1], linex(t3)));
\end{quotation}
This program gives the following parametrization of $S$:
\begin{eqnarray*}
x_1(y_1,y_2)&=&-\frac{y_1^3-2y_1^2-y_1^2y_2+3y_1+2y_1y_2+y_2^2y_1+y_2^2}{y_1^3-2y_1^2-y_1^2y_2+3y_1+2y_1y_2+y_2^2y_1-3-3y_2-2y_2^2}, \\
x_2(y_1,y_2)&=& \frac{2y_1^2+y_1^2y_2-y_2^2y_1-2y_1y_2-3y_1+3+3y_2+y_2^3+2y_2^2}{y_1^3-2y_1^2-y_1^2y_2+3y_1+2y_1y_2+y_2^2y_1-3-3y_2-2y_2^2}, \\
x_3(y_1,y_2)&=& -\frac{y_1^2y_2-y_1^2-2y_1y_2-y_2^2y_1+y_2^3+2y_2^2+3y_2}{y_1^3-2y_1^2-y_1^2y_2+3y_1+2y_1y_2+y_2^2y_1-3-3y_2-2y_2^2}.
\end{eqnarray*}

\subsection{A general construction with MAPLE}

Following Section \ref{s:construction} we consider a triple 
$(\ell_1,\ell_2,m)$ where $\ell_1$, $\ell_2$ and $m$ are three lines in $\P^3$ such that 
$\ell_1$ and $\ell_2$ are skew and $m$ intersects both $\ell_1$ and $\ell_2$. 
We say that a cubic surface admits a triple $(\ell_1,\ell_2,m)$ 
if the three lines $\ell_1$, $\ell_2$, $m$ are contained in $S$.
If $\K$ is a subfield of $\C$ we say that the triple $(\ell_1,\ell_2,m)$ is 
defined over $\K$ if $\ell_1\cup \ell_2\cup m$  is defined over $\K$, i.e. $m$ is defined over $\K$ and 
$\ell_1$ and $\ell_2$ are conjugate on a quadratic extension $\L$ of $\K$.

For our purposes it is enough to consider the case $m=(t,0,0)$, $\ell_1=(a+b_1t,b_2t,b_3t)$ and $\ell_2=(c+d_1t,d_2t,d_3t)$ 
(which is the general case 
up to a projective transformation). First we compute the space $V$ of  cubic surfaces admitting the triple 
$(\ell_1,\ell_2,m)$. If the triple $(\ell_1,\ell_2,m)$ is defined over $\K$ so will the space $V$. 
We can easily compute the space $V$ .

\begin{quotation}
\sffamily
\noindent 
with(PolynomialTools):
\noindent
\# Consider a Generic cubic: \\
\noindent
f:=(x1,x2,x3)$\to$A*x1\^{}3+B*x2\^{}3+C*x3\^{}3+D*x1\^{}2*x2+E*x1\^{}2*x3+F*x2\^{}2*x1\\
+G*x2\^{}2*x3+H*x3\^{}2* x1+  J*x3\^{}2*x2+K*x1*x2*x3+L*x1\^{}2+M*x2\^{}2+N*x3\^{}2+\\ 
O*x1*x2+P*x1*x3+Q*x2*x3+R*x1+S*x2+T*x3+U;
\\
\noindent
\# Choose two disjoint lines $\ell_1$ and $\ell_2$: \\
\noindent
line1:=t$\to$[a+b1*t,b2*t,b3*t];\\
\noindent
line2:=t$\to$[c+d1*t,d2*t,d3*t]; \\
\noindent
\# Choose a line m that intersect both $\ell_1$ and $\ell_2$: \\
\noindent
linem:=t$\to$[t,0,0]; \\
\noindent
\# Compute the Coefficient Vector of $\ell_1$, $\ell_2$ and $m$: \\
\noindent
E1:=expand(subs({x1=line1(t)[1],x2=line1(t)[2],x3=line1(t)[3]},f(x1,x2,x3))); \\
\noindent
E2:=expand(subs({x1=line2(t)[1],x2=line2(t)[2],x3=line2(t)[3]},f(x1,x2,x3))); \\
\noindent
Em:=expand(subs({x1=linem(t)[1],x2=linem(t)[2],x3=linem(t)[3]},f(x1,x2,x3))); \\
\noindent
V1:=CoefficientVector(collect(E1,t),t); \\
\noindent
V2:=CoefficientVector(collect(E2,t),t); \\
\noindent
Vm:=CoefficientVector(collect(Em,t),t); \\
\noindent
\# Compute the variety of the cubics passing through the lines  $\ell_1$, $\ell_2$ and $m$. \\
\noindent 
Sols:=solve({V1[1]=0,V1[2]=0,V1[3]=0,V1[4]=0,V2[1]=0,V2[2]=0,V2[3]=0,V2[4]=0, \\
Vm[1]=0, Vm[2]=0,Vm[3]=0,Vm[4]=0},{A,B,C,D,E,F,G,H,J,K,L,M,N,O,P,Q,R,S,T,U}): \\
\noindent
Cubic:=subs(Sols,f(x1,x2,x3));
\end{quotation}

This program gives a description of the space of cubics passing through the lines $\ell_1$, $\ell_2$ and $m$. 
One can proceed by computing a parametrization of a general cubic in $V$  with the method proposed in 
Section \ref{s:construction} as follows.

\begin{quotation}
\sffamily
\noindent
\# Consider the plane H=(y1,y1+y2,1-y2) \\
\noindent
linex := t $\to$ [y1,y2,0]+[t, n*t, m*t];
\\
\noindent
\# Solve\\
\noindent
Sol := solve({line1(r)[1] = linex(t1)[1], line1(r)[2] = linex(t1)[2], line1(r)[3] = linex(t1)[3], line2(s)[1] = linex(t2)[1],  
line2(s)[2] = linex(t2)[2], line2(s)[3] = linex(t2)[3], subs({x1=linex(t3)[1], x2=linex(t3)[2], x3=linex(t3)[3]}, Cubic)=0, 
t1$<>$t3,t2$<>$t3}, [r, s, n, m, t1, t2, t3]):\\
\noindent
\# Now the parametrization is \\
\noindent
parametrization := factor(subs(Sol[1], linex(t3))): \\
\end{quotation}

This gives a parametrization by cubic polynomials (in $y_1$ and $y_2$) as it is claimed in Theorem \ref{construction}.

\begin{eqnarray*}
x_1(y_1,y_2)&=&-\frac{-2 T a^2 b_2^3 d_3^2 d_2 y_1^2-T a^2 b_2^2 c b_1  d_2 d_3^2y_2-3 T a^2 b_2^2 c d_3 b_3 d_2^2 y_1+\ldots }
{2 C a b_2 c d_3^3 b_1 b_3 y_2^3+2 K a^2 b_2 c b_3^2  d_2^3 y_1y_2-2 K a^2 b_2 c  d_2^2 b_3^2 d_1y_2^2+\ldots }, \\
x_2(y_1,y_2)&=& \frac{2 d_2 b_2^2 c T b_1 d_3^2 y_1 y_2^2-2 d_2 b_2^2  a c d_3^2 P b_1 y_1y_2^2-2 d_2 b_2^2  a T b_1  d_3^2y_1y_2^2+\ldots }
{2 C a b_2 c d_3^3 b_1 b_3 y_2^3+2 K a^2 b_2 c b_3^2  d_2^3 y_1y_2-2 K a^2 b_2 c  d_2^2 b_3^2 d_1y_2^2+\ldots }, \\
x_3(y_1,y_2)&=& -\frac{ b_3 d_3 a b_2^2 c^3 D d_1 d_2y_2^2-3  b_3 d_3 a^2 b_2^2 c^2 D d_1 d_2y_2^2+2 b_3 d_3 a b_2^2 c^3 D d_2^2 y_1y_2 +\ldots}
{2 C a b_2 c d_3^3 b_1 b_3 y_2^3+2 K a^2 b_2 c b_3^2  d_2^3 y_1y_2-2 K a^2 b_2 c  d_2^2 b_3^2 d_1y_2^2+\ldots }, \\
\end{eqnarray*}

\end{document}